\documentclass[reqno, 12pt, a4letter]{amsart}
\usepackage{amsmath,amsxtra,amssymb,latexsym,amscd,amsthm}
\usepackage[mathscr]{euscript}
\usepackage{mathrsfs}
\usepackage[english]{babel}
\usepackage[active]{srcltx}
\usepackage{enumerate}
\usepackage{color}
\usepackage{mathtools,stmaryrd}
\usepackage[T1]{fontenc}
\usepackage{tikz}
\usepackage{tikz-cd}
\usepackage[utf8]{inputenc}

\newtheorem*{MainTheorem}{Main Theorem}
\newtheorem {theorem}{Theorem}[section]

\newtheorem {corollary}[theorem]{Corollary}
\newtheorem {lemma}[theorem]{Lemma}

\newtheorem {definition}[theorem]{Definition}

\newtheorem {proposition}[theorem]{Proposition}

\newcommand{\Int}{{\rm int}}

\def\ees{{\accent"5E e}\kern-.385em\raise.2ex\hbox{\char'23}\kern-.08em}
\def\EES{{\accent"5E e}\kern-.5em\raise.8ex\hbox{\char'23 }}
\def\ow{o\kern-.42em\raise.82ex\hbox{\vrule width .12em height .0ex depth .075ex \kern-0.16em \char'56}\kern-.07em}
\def\OW{o\kern-.460em\raise1.36ex\hbox{
\vrule width .13em height .0ex depth .075ex \kern-0.16em
\char'56}\kern-.07em}
\def\DD{D\kern-.7em\raise0.4ex\hbox{\char '55}\kern.33em}

\title{Stability of closedness of semi-algebraic sets under continuous semi-algebraic mappings}

\author[S. T. \DD inh]{S\~i Ti\d{\^e}p \DD inh$^\dag$}
\address{Institute of Mathematics, VAST, 18, Hoang Quoc Viet Road, Cau Giay District 10307, Hanoi, Vietnam}
\email{dstiep@math.ac.vn}

\author{Zbigniew Jelonek$^\ddagger$}
\address{Institute of Mathematics, Polish Academy of Sciences, \'Sniadeckich 8, 00-656 Warsaw, Poland}
\email{najelone@cyf-kr.edu.pl}

\author[T. S. Ph\d{a}m]{Ti\'{\^e}n-S\OW n Ph\d{a}m$^*$}
\address{Department of Mathematics, Dalat University, 1 Phu Dong Thien Vuong, Dalat, Vietnam}
\email{sonpt@dlu.edu.vn}

\subjclass{Primary 14P10, 58A35 Secondary 14P15, 32C05, 58A07}

\keywords{Closedness, tangent cones at infinity, semi-algebraic sets/mappings, linear mappings, stability}

\date{ \today}

\begin{document}

\begin{abstract}
Given a closed semi-algebraic set $X \subset \mathbb{R}^n$ and a continuous semi-algebraic mapping $G \colon X \to \mathbb{R}^m,$ it will be shown that 
there exists an open dense semi-algebraic subset $\mathscr{U}$ of $L(\mathbb{R}^n, \mathbb{R}^m),$ the space of all linear mappings from $\mathbb{R}^n$ to $\mathbb{R}^m,$ such that for all $F \in \mathscr{U},$ the image $(F + G)(X)$ is a closed (semi-algebraic) set in $\mathbb{R}^m.$
To do this, we study the tangent cone at infinity $C_\infty X$ and the set $E_\infty X \subset C_\infty X$ of (unit) exceptional directions at infinity of $X.$ Specifically we show that the set $E_\infty X$ is nowhere dense in $C_\infty X \cap \mathbb{S}^{n - 1}.$ 
\end{abstract}

\maketitle

\pagestyle{plain}

\section{Introduction}

Exploring generic properties of a class of objects such as sets and/or mappings, $\dots$ is a fundamental problem in Theory of Singularities.  In the paper~\cite{Mather1973} (see also \cite{Ichiki2018}), considering the class of {\it projections}, i.e., the restriction of a linear surjective mapping from a vector space $V$ into a vector space $Y$ to a submanifold $X$ of $V,$  John~N.~Mather asked what properties of mappings are true for a generic projection and he gave answers to several special cases of this question. Motivated by this work and various problems in analysis and optimization \cite{Kim2019,Kim2020}, we consider the question if preserving closedness is a generic property of continuous semi-algebraic mappings. Namely, we study when the image of a closed semi-algebraic set under a continuous semi-algebraic mapping is closed and if the closedness is stable under small linear perturbations. The closedness of such images is of significance in analysis, since it allows one to keep lower semi-continuity of functions and to assure the existence of solutions to various extremum problems (see, for example, \cite{Auslender2003,Pataki2001}).

It is well-known that the closedness of the image of a closed convex set under a linear mapping is not necessarily preserved under small perturbations of the linear mapping (see, for example, \cite[Examples~2.1~and~2.2]{Borwein2009}). On the other hand, it was shown recently in
\cite{Dinh2020-1} (see also \cite{Borwein2009, Borwein2010}) that for a given closed convex set $X$ in $\mathbb{R}^n,$ the set
$$L(\mathbb{R}^n, \mathbb{R}^m) \ \setminus \ \Int(\{F\in L(\mathbb{R}^n, \mathbb{R}^m) \ : \ F(X) \text{ is closed}\})$$
is $\sigma$-porous in $L(\mathbb{R}^n, \mathbb{R}^m)$-the space of all linear mappings from $\mathbb{R}^n$ to $\mathbb{R}^m,$ i.e. small with regard to both measure and category. 

The aim of this paper is to prove a similar result by considering a class of sets not necessarily being convex, which is the class of semi-algebraic sets. Precisely, with the definitions given in the next section, the main result of this paper is as follows.

\begin{MainTheorem}
Let $X \subset \mathbb{R}^n$ be a closed semi-algebraic set and let $G \colon X \to \mathbb{R}^m$ be a continuous semi-algebraic mapping. Then the set
$$\{F\in L(\mathbb{R}^n, \mathbb{R}^m) \ : \ (F + G)(X) \text{ is closed}\}$$
is semi-algebraic and contains an open dense semi-algebraic subset of $L(\mathbb{R}^n, \mathbb{R}^m).$
\end{MainTheorem}

Note that in the setting of Semi-Algebraic Geometry, a semi-algebraic subset of $\mathbb{R}^N$ is open dense if and only if  its complement is $\sigma$-porous.  Moreover, although the results presented here still hold for sets/mappings definable in some o-minimal structure (see \cite{Dries1996, Dries1998} for more on the subject), we prefer to work with semi-algebraic sets/mappings for simplicity.  We also emphasize that the technique used in~\cite{Dinh2020-1} for the convex case can not be applied to prove Main Theorem.

In the case where $X$ is a smooth irreducible affine variety in $\mathbb C^n$ of dimension~$\leqslant m$ and $G\colon X\to \mathbb C^m$ is a polynomial mapping, a similar result can be found in~\cite[Theorem~8.2]{Farnik2020}. The proof of this result, which only deals  with finite mappings, exploits the fact that a projection of an affine variety of dimension not greater than $m$ on $\mathbb C^m$ is generically proper (see \cite[Theorem~2.1 and Corollary~2.2]{Farnik2020}). On the other hand, our main result treats continuous semi-algebraic mappings on closed semi-algebraic sets with no constraint on the dimension; in addition, the proof given here is different from that given in~\cite{Farnik2020} . 

The rest of the paper is organized as follows. In Section~\ref{SectionPreliminary}, we present some preliminaries of Semi-Algebraic Geometry which will be used later. 
The notion and some basic properties of of tangent cone at infinity will be provided in Section~\ref{Section3}. 
Finally, the proof of Main Theorem will be completed in Section~\ref{Section5}.

\section{Preliminaries} \label{SectionPreliminary}

\subsection{Notation}
Let $\mathbb{R}^n$ denote the Euclidean space of dimension $n.$ The corresponding inner product (resp., norm) in $\mathbb{R}^n$  is denoted by $\langle x, y \rangle$ for any $x, y \in \mathbb{R}^n$ (resp., $\| x \| := \sqrt{\langle x, x \rangle}$ for any $x \in \mathbb{R}^n$).  We will denote by $\mathbb B^n_r$ and $\mathbb S^{n-1}_r(x),$ respectively, the open ball and the sphere of radius $r$ centered at the origin in $\mathbb R^n.$ For simplicity, we write $\mathbb B^n$ and $\mathbb S^{n-1}$ if $r = 1.$

The closure and the interior of $X\subset\mathbb R^n$ are denoted respectively by $\overline X$ and $\Int (X).$

Let $L(\mathbb R^n,\mathbb R^m)$ denote the set of all linear mappings from $\mathbb{R}^n$ to $\mathbb{R}^m.$ For any linear mapping $F\in L(\mathbb R^n,\mathbb R^m),$ we may identify $F$ with the matrix of $F$ in the canonical bases of $\mathbb R^n$ and $\mathbb R^m,$ and so we can identify $L(\mathbb R^n,\mathbb R^m)$ with $\mathbb{R}^{m \times n}$.

\subsection{Semi-algebraic geometry}

We recall some notions and results of Semi-Algebraic Geometry, which can be found in \cite{Benedetti1990, Bochnak1998, Coste2000-1, Dries1996, Dries1998}.

\begin{definition}{\rm 
\begin{enumerate}[{\rm (i)}]
  \item A subset of $\mathbb{R}^n$ is called {\em semi-algebraic} if it is a finite union of sets of the form 
$$\{x \in \mathbb{R}^n: \ f(x) = 0 ;\ g_i(x) > 0, i = 1, \ldots, k\}$$
where $f$ and all $g_{i}$ are polynomials.
 \item Let $X \subset \mathbb{R}^n$ and $Y \subset \mathbb{R}^m$ be semi-algebraic sets. A mapping $F \colon X \to Y$ is said to be {\em semi-algebraic} if its graph 
$$\{(x, y) \in X \times Y  : \ y = F(x)\}$$
is a semi-algebraic subset of $\mathbb{R}^n\times\mathbb{R}^m.$
\end{enumerate} }
\end{definition}

A major fact concerning the class of semi-algebraic sets is the following theorem.

\begin{theorem}[{Tarski--Seidenberg,  \cite[Theorem 2.3.4]{Benedetti1990}, \cite[Proposition~2.2.7]{Bochnak1998}, \cite[Proposition~1.7]{HaHV2017}}] \label{TarskiSeidenbergTheorem1}
The image of a semi-algebraic set by a semi-algebraic mapping is semi-algebraic.
\end{theorem}

Moreover, we have the following properties:
\begin{enumerate}[{\rm (i)}]
\item The class of semi-algebraic sets is closed with respect to Boolean operators, taking Cartesian product, closure, interior and inverse image under semi-algebraic mappings;
\item A composition of semi-algebraic mappings is a semi-algebraic mapping;
\item The  inverse image of  a  semi-algebraic set  under  a  semi-algebraic mapping is semi-algebraic;
\end{enumerate}

\subsection{Stratification of semi-algebraic sets}

\begin{definition}{\rm 
A {\em semi-algebraic stratification} of a semi-algebraic subset $X$ of $\mathbb R^n$ is a finite partition $\mathcal S:=\{X_\alpha\}_{\alpha \in I}$ of $X$ such that:
\begin{enumerate}[\ $\bullet$]
\item Each $X_\alpha$ (called a {\em stratum of the stratification} or, briefly, a {\em stratum of $X$}) is a connected semi-algebraic analytic submanifold of $\mathbb R^n$.

\item The following {\em frontier condition} holds: If ${X_\alpha}\cap \overline{X}_\beta\ne\emptyset$, then ${X_\alpha}\subset \overline{X}_\beta.$
\end{enumerate} }
\end{definition}

The following is well-known (see, for example, \cite[Proposition~2.5.1]{Benedetti1990} and \cite[Proposition~9.1.8]{Bochnak1998}).

\begin{proposition}\label{CD1}
Every semi-algebraic set admits a semi-algebraic stratification.
\end{proposition} 

Let $X\subset\mathbb R^n$ be a semi-algebraic set and let $\mathcal S:=\{X_\alpha\}_{\alpha \in I}$ be a semi-algebraic stratification of $X.$ We define the {\em dimension} of $X$ by
$$\dim X:=\max\{\dim X_\alpha:\ \alpha\in I\}.$$
It is not hard to check that this definition of dimension does not depend on the stratification of $X.$ For convenience, set $\dim\emptyset=-1$.  
Furthermore, we have the following properties of the dimension; see, for example, \cite[Section~2.8]{Bochnak1998}, \cite[Proposition~1.4]{HaHV2017}, \cite[Property~4.7]{Dries1996}.


\begin{lemma} \label{DimensionLemma}
If $X$ be a non-empty semi-algebraic set in $\mathbb{R}^n,$ then $\dim(\overline{X}\setminus X)<\dim X.$ In particular, $\dim\overline{X}=\dim X.$
\end{lemma}

\begin{lemma}\label{DenseIsOpen} 
Let $X\subset\mathbb R^n$ be a semi-algebraic set. The following statements are equivalent.
\begin{enumerate}[{\rm (i)}]
\item $X$ is dense in $\mathbb R^n;$
\item $X$ contains an open dense semi-algebraic subset of $\mathbb R^n;$
\item $\dim(\mathbb{R}^n \setminus X) < n.$
\end{enumerate}
\end{lemma}

\begin{proof} 
(i) $\Rightarrow$ (ii).
Let $\mathcal S = \{X_\alpha\}_{\alpha \in I}$ be a semi-algebraic stratification of ${X}.$ Let $Y$ be the union of strata of dimension $n.$
Then it is easy to see that $Y$ has the desired properties.

(ii) $\Rightarrow$ (i). Clearly.

(i) $\Rightarrow$ (iii).  By contradiction, suppose that $\dim (\mathbb{R}^n \setminus X) = n.$ 
The set $\mathbb{R}^n \setminus X$ is semi-algebraic (because $X$ is semi-algebraic) and hence it contains a nonempty open set. This implies that $\overline{X} \ne \mathbb{R}^n,$ which contradicts the assumption.

(iii) $\Rightarrow$ (i).  
Take any $x \not \in X.$ By assumption, $U \cap {X} \ne \emptyset$ for all open sets $U$ containing $x.$ This implies that $x \in \overline{X}$ and hence $X$ is dense in $\mathbb{R}^n.$
\end{proof}

In the sequel we will make use of the following result. 

\begin{theorem}[{Hardt, \cite{Hardt1980}}] \label{HardtTheorem}
Let $X$ and $Y$ be respectively semi-algebraic sets in $\mathbb{R}^n$ and $\mathbb{R}^m$, $f \colon X \rightarrow Y$ be a continuous semi-algebraic mapping.
Then there exists a partition of $Y$ into finitely many semi-algebraic subsets $Y_i, i = 1, \ldots, p,$ such that $f$ is semi-algebraically trivial over each $Y_i,$ i.e., there exists a semi-algebraic set $F_i \subset \mathbb{R}^n$ and a semi-algebraic homeomorphism
$h_i \colon f^{-1} (Y_i) \rightarrow Y_i \times F_i$ 
such that the following diagram commutes
$$\begin{tikzcd}
f^{-1}(Y_i) \arrow {r}{h_i}
\arrow{rd}{f}
& Y_i \times F_i \arrow {d}{\pi} \\
& Y_i
\end{tikzcd}$$
where $\pi$ is the projection on the first factor.
\end{theorem}

As a consequence of Theorem~\ref{HardtTheorem}, we obtain the following useful corollaries. 

\begin{corollary}\label{Everywhere} 
Let $\mathscr{U} \subset \mathbb R^n\times\mathbb R^m$ be an open dense semi-algebraic set. Then there is a set $\mathscr{V}\subset\mathbb R^m$ containing an open dense semi-algebraic set in $\mathbb R^m$ such that for any $y \in \mathscr{V},$ the semi-algebraic set
$$\mathscr{U}_y:=\{x\in\mathbb R^n:\ (x,y)\in \mathscr{U}\}$$
contains an open dense semi-algebraic set in $\mathbb R^n.$
\end{corollary}

\begin{proof} 
Let $f\colon(\mathbb R^n\times\mathbb R^m)\setminus \mathscr{U} \to\mathbb R^m$ be the projection $(x,y)\mapsto y$. 
In light of Theorem~\ref{HardtTheorem}, we can write  $\mathbb R^m = \bigcup_{i = 1}^p Y_i$ as a disjoint union of semi-algebraic sets $Y_i \subset \mathbb R^m$ such that $f^{-1}(Y_i)$ is semi-algebraically homeomorphic to $f^{-1}(y_i)\times Y_i,$ for each $i$ and any $y_i\in Y_i.$ Let 
$$\displaystyle \mathscr{V} :=\bigcup_{\substack{i=1,\dots,p;\\\dim Y_i=m}}Y_i.$$
Then $\mathscr{V}$ is a semi-algebraic set in $\mathbb R^m$ satisfying $\dim(\mathbb R^m\setminus \mathscr{V})<m.$ By Lemma~\ref{DenseIsOpen}, $\mathscr{V}$ contains an open dense semi-algebraic subset of $\mathbb R^m.$

Assume for contradiction that there exists $y \in \mathscr{V}$ such that $\mathscr{U}_{y}$ does not contain an open dense semi-algebraic set in $\mathbb R^n.$ It follows that $\mathbb R^n\setminus \mathscr{U}_{y}$ contains an open set in $\mathbb R^n,$ which yields that
$$\dim (f^{-1}(y))=\dim((\mathbb R^n\setminus \mathscr{U}_{y})\times\{y\}) = n.$$ 
Let $i\in\{1,\dots,p\}$ be such that $y\in Y_i.$ 
Since $f^{-1}(Y_i)$ is semi-algebraically homeomorphic to $f^{-1}(y)\times Y_i,$ it follows that 
$$\dim(f^{-1}(Y_i))=\dim(f^{-1}(y))+\dim Y_i=n+m,$$
Hence, $f^{-1}(Y_i)$ contains an open subset of $\mathbb R^n\times\mathbb R^m$. This contradicts the assumption that the set $\mathscr{U}$ is dense in $\mathbb R^n\times\mathbb R^m.$ The corollary is proved.
\end{proof}

\begin{corollary}\label{DimInfinity} 
Let $X\subset\mathbb R^n$ be an unbounded semi-algebraic set. For $R>0$ large enough, the following statements hold:
\begin{enumerate}[{\rm (i)}]
\item the dimension of $X\setminus\mathbb B^n_R$ is constant and is called the {\em dimension at infinity} of $X$, denoted by $\dim_\infty X;$
\item the dimension of $X\cap\mathbb{S}^{n - 1}_R$ is constant and equals to $\dim_\infty X - 1.$
\end{enumerate}
\end{corollary}
\begin{proof} 
Consider the continuous semi-algebraic function $f \colon X \to \mathbb R, x \mapsto \|x\|.$ In light of Theorem~\ref{HardtTheorem}, we can find a constant $R > 0$ such that $f$ is semi-algebraically trivial over the interval $(R, +\infty).$ From this, the desired conclusion follows easily.
\end{proof}

\subsection{Sard's theorem with parameter}

Next we state a semialgebraic version of Sard's theorem with the
parameter  in a simplified form which is sufficient for the
applications studied here. Given a differentiable map between differentiable
manifolds $f \colon X \rightarrow Y,$ a point $y \in Y$ is called a
{\em regular value} \index{regular value} for $f$ if either $f^{-1}(y)
= \emptyset$ or the derivative map $Df(x) \colon T_xX \rightarrow
T_yY$ is surjective at every point $x \in f^{-1}(y),$ where $T_x X$
and $T_yY$ denote the tangent spaces of $X$ at $x$ and of $Y$ at $y,$
respectively. A point $y \in Y$ that is not a regular value of $f$ is
called a {\em critical value.} The following result is also called Thom's weak transversality theorem.

\begin{theorem}[Sard's theorem with parameter] \label{SardTheorem}
Let $f \colon X \times \mathscr{P} \rightarrow Y$ be a differentiable
semialgebraic map between semialgebraic submanifolds.
If $y \in Y$ is a regular value of $f,$ then there exists a
semialgebraic set $\Sigma \subset \mathscr{P}$
of dimension smaller than the dimension of $\mathscr{P}$ such that$,$ for every $p
\in \mathscr{P} \setminus \Sigma,$ $y$ is a regular value of the map
$f_p \colon X \rightarrow Y, x \mapsto f(x, p).$
\end{theorem}
\begin{proof}
For a proof, we refer the reader to \cite[Theorem~1.3.6]{Goresky1988}, \cite[The Transversality Theorem, page 68]{Guillemin1974} or \cite[Theorem~1.10]{HaHV2017}.
\end{proof}

\section{Tangent cones at infinity}\label{Section3}

In this section, we define and present some properties of tangent cones at infinity of unbounded semi-algebraic sets in $\mathbb R^n.$ 
First of all, let us agree to call a set $C \subset \mathbb{R}^n$ a {\em cone} if whenever $x \in C,$ then $tx \in C$ for all $t \geqslant 0$.

\begin{definition}{\rm 
By the {\em tangent cone at infinity} (known also as the {\em asymptotic cone}) of  a subset $X$ of $\mathbb{R}^n$ we mean the set
\begin{eqnarray*}
C_\infty X := 
\left\{
\begin{array}{llll}v \in \mathbb R^n:& \textrm{there exist sequences } x^k\in X \textrm{ and } t_k\in (0,+\infty) \textrm{ such that }\\
& x^k\to \infty \textrm{ and } t_k x^k\to v \textrm{ as }k\to\infty
\end{array}\right\}.
\end{eqnarray*} }
\end{definition}

The following lemma is a version at infinity of \cite[Lemma~1.2]{Kurdyka1989} (see also \cite[Corollary~2.18]{Fernandes2020}).

\begin{lemma}\label{Lemma32} 
Let $X$ be an unbounded semi-algebraic subset of $\mathbb R^n.$ Then $C_\infty X$ is a nonempty closed semi-algebraic cone of dimension at most $\dim_\infty X,$ where $\dim_\infty X$ is defined in Corollary~\ref{DimInfinity}.
\end{lemma}

\begin{proof}
Since the set $X$ is unbounded, there exists a sequence $x^k \in X$ with $x^k\to \infty.$ Then for all sufficiently large $k,$ we have $t_k := \|x^k\|^{-2} \in (0, +\infty).$ Moreover,  $t_k x^k \to 0 \textrm{ as }k\to\infty.$ By definition, $0 \in C_\infty X$ and so $C_\infty X$ is nonempty.

Let $v \in C_\infty X.$ By definition, there exist sequences $x^k \in X$ and $t_k \in (0, +\infty)$ such that 
$$x^k\to \infty \quad \textrm{ and } \quad  t_k x^k\to v \textrm{ as }k\to\infty.$$
Take any $\lambda \in (0, +\infty).$ Observe that $\lambda t_k \in (0, +\infty)$ and $(\lambda t_k) x^k \to \lambda v$ as $k \to \infty.$ By definition, then $\lambda v  \in C_\infty X.$ Hence $C_\infty X$ is a cone.

For each $R > 0,$ define
$$A_R :=\{(tx, t) \in \mathbb{R}^n \times (0,+\infty): \ x \in X \setminus \mathbb{B}^n_R\}.$$
Clearly, $A_R$ is a nonempty semi-algebraic set.
Moreover 
\begin{equation}\label{sub}C_\infty X \times \{0\} \subset \overline{A}_R \cap (\mathbb{R}^n \times \{0\}) \subset \overline{A}_R \setminus A_R.\end{equation} 
Let us prove that 
\begin{equation}\label{sup}C_\infty X \times \{0\} \supset \overline{A}_R \cap (\mathbb{R}^n \times \{0\}).\end{equation}
For this, let $(v,0)\in \overline{A}_R \cap (\mathbb{R}^n \times \{0\}).$ 
If $v=0,$ then clearly $(v,0)\in C_\infty X \times \{0\}.$
So assume that $v\ne 0.$
By definition, there are sequences $x^k\in X \setminus \mathbb{B}^n_R$ and $t_k\in (0,+\infty)$ such that 
$$t_k\to 0\ \text{ and }\ t_k x^k\to v.$$
This, together with the fact that $v\ne 0$, implies that $x^k\to\infty.$ 
Thus, by definition, $v\in C_\infty X $ and so $(v,0)\in C_\infty X \times \{0\}.$
Consequently,~\eqref{sup} holds.
From~\eqref{sub} and~\eqref{sup}, we have 
$$C_\infty X \times \{0\} = \overline{A}_R \cap (\mathbb{R}^n \times \{0\}).$$
As $C_\infty X \times \{0\}$ is the intersections of two closed semi-algebraic sets, it follows that $C_\infty X \times \{0\}$ and so $C_\infty X$ are closed semi-algebraic sets.

Finally, observe that $A_R$ is homeomorphic to $(X \setminus \mathbb{B}^n_R)\times (0,+\infty).$ Hence, for all $R$ large enough,
$$\dim A_R = \dim_\infty X + 1,$$
which, combined with Lemma~\ref{DimensionLemma} and~\eqref{sub}, yields
$$\dim C_\infty X  \leqslant \dim (\overline{A}_R \setminus A_R) < \dim A_R = \dim_\infty X + 1.$$
This completes the proof of the lemma.
\end{proof}

\section{Proof of the main theorem}\label{Section5}

The goal of this section is to give a proof of Main Theorem. 
First of all, we provide a sufficient condition for the restriction of a linear mapping on a closed semi-algebraic set to be closed.

\begin{lemma}\label{Theorem41} 
Let $X\subset\mathbb R^n$ be a closed semi-algebraic set and let $F \colon \mathbb{R}^n \to \mathbb{R}^m$ be a linear mapping. Set $ D := C_\infty X\cap\mathbb S^{n-1}$ and assume that $\ker F\cap D=\emptyset$. Then $F(X)$ is closed. 
\end{lemma}

\begin{proof}
Let $y \in \overline{F(X)}.$ Then there is a sequence $x^k\in X$ such that $\displaystyle\lim_{k\to\infty} F(x^k)=y.$ We will show that $x^k$ has a convergent subsequence. For contradiction, assume that $x^k\to\infty$. Passing to a subsequence if necessary, we may suppose that $\displaystyle\lim_{k\to\infty}\frac{x^k}{\|x^k\|} = u \in D.$ Since $D$ is compact and $\ker F \cap D = \emptyset,$ it holds that $\displaystyle\|F(u)\|\geqslant\min_{x\in D}\|F(x)\|>0$ and so
$$\|y\| = \lim_{k\to\infty}\|F(x^k)\|=\lim_{k\to\infty}\|x^k\|\left\|F\left(\frac{x^k}{\|x^k\|}\right)\right\| = +\infty,$$
which is impossible. Therefore the sequence $x^k$ has a cluster point, say $x$. Clearly $x\in X$ and $F(x)=y$, so $y\in F(X)$.
\end{proof}

We are now in a position to prove the main result of the paper.

\begin{proof}[Proof of Main Theorem] 

By Theorem~\ref{TarskiSeidenbergTheorem1},  it is not hard to check that the set 
$$\{F\in L(\mathbb{R}^n, \mathbb{R}^m) \ : \ (F + G)(X) \text{ is closed}\}$$ 
is semi-algebraic so it remains to show that it contains an open dense semi-algebraic subset of $L(\mathbb{R}^n, \mathbb{R}^m).$ 
The proof is divided into two cases. First of all, we consider the case $G\equiv 0$.

\subsubsection*{Case 1: Linear case}

If $n\leqslant m$, then for any $F\in \mathcal{A}$, $F$ is a linear isomorphism from $\mathbb R^n$ onto $F(\mathbb{R}^n)$ and so $F(X)$ is closed. 
Thus from now on, suppose that $n > m.$ 
We consider two sub-cases.

\subsubsection*{Case $1.1$: $d\leqslant m$}\

Set $D := C_\infty X\cap\mathbb S^{n-1}$.
In view of Proposition~\ref{CD1}, let $\mathcal S=\{D_\alpha\}_{\alpha\in I}$ be a semi-algebraic stratification of $D.$
Recall that we identify $L(\mathbb R^n,\mathbb R^m)$ with $\mathbb{R}^{m \times n}$. For each $\alpha \in I$ define the semi-algebraic mapping
$$\Phi_\alpha \colon \mathbb{R}^{m \times n} \times D_\alpha \to \mathbb  R^m, \quad (A, v) \mapsto Av.$$
Write $A=(a_{ij})_{\substack{i=1,\dots,m\\ j=1,\dots,n}}.$ Since $D_\alpha \subset D \subset \mathbb{S}^{n - 1},$ for each $v \in D_\alpha$ there exists an index $i_0,\ 1 \leqslant i_0 \leqslant n,$ such that $v_{i_0} \ne 0,$ and so the Jacobian matrix $D\Phi_\alpha$ of $\Phi_\alpha$ at $(A, v)$ contains the following diagonal matrix
$$\frac{\partial\Phi_\alpha}{\partial(a_{1i_0},\dots,a_{mi_0})} = 
\begin{pmatrix}
v_{i_0} & 0 & \cdots & 0\\
0 & v_{i_0} & \cdots & 0\\
\vdots & \vdots & \ddots & \vdots \\
0 & 0 & \cdots & v_{i_0}
\end{pmatrix},$$
which has rank $m.$ Hence $D\Phi_\alpha$ is of maximal rank $m$ at every point $(A,v) \in \mathbb{R}^{m \times n} \times D_\alpha.$ Consequently, $0$ is a regular value of the mapping $\Phi_\alpha.$ By Theorem~\ref{SardTheorem},
$$\mathscr{U}_\alpha := \{A\in \mathbb{R}^{m \times n} : 
0 \textrm{ is a regular value of the mapping } \Phi_\alpha(A,\cdot) \colon D_\alpha \to \mathbb{R}^m \}$$ 
contains an open dense semi-algebraic set in $\mathbb{R}^{m \times n}.$

On the other hand, we have
$$\dim D_\alpha \leqslant \dim D < \dim C_\infty X \leqslant d \leqslant m,$$
where the third inequality follows from Lemma~\ref{Lemma32}. Hence, for each fixed $A\in \mathbb{R}^{m \times n}$, $0$ is a regular value of the mapping $\Phi_\alpha(A,\cdot) \colon D_\alpha \to \mathbb{R}^m$ if and only if 
$$\{v \in D_\alpha : \ \Phi_\alpha(A, v) = 0\}  = \emptyset.$$
Therefore, we can write
$$\mathscr{U}_\alpha = \{F\in L(\mathbb R^n,\mathbb R^m):\ \ker F\cap D_\alpha=\emptyset\}.$$
Let $$\displaystyle \mathcal{A}_1 := \bigcap_{\alpha\in I}\mathscr{U}_\alpha = \{F\in L(\mathbb R^n,\mathbb R^m):\  \ker F\cap D=\emptyset\}.$$ 
Then $\mathcal{A}_1$ contains an open dense semi-algebraic set in $L(\mathbb R^n,\mathbb R^m).$ 
In view of Lemma~\ref{Theorem41}, $F(X)$ is closed for each $F\in\mathcal {A}_1$.
This end the proof for the case $d\leqslant m.$

\subsubsection*{Case $1.2$: $d> m$}\

The proof for this case is proceeded by induction on $d.$
There are two sub-cases to consider.

\subsubsection*{Case $1.2.1$: $d=n$}\

It is clear that the interior $\Int(X)$ of X is non-empty and open in $\mathbb R^n$.
Set $\partial X := X\setminus\Int(X).$ Then $\partial X$ is a closed semi-algebraic set of dimension less than $d.$ 
By the inductive assumption, there is an open dense semi-algebraic $\mathcal A_2\subset L(\mathbb R^n,\mathbb R^m)$ such that $F(\partial X)$ is closed for any $F\in\mathcal A_2.$ Let
\begin{equation*}
\mathcal{A} := \{F\in L(\mathbb{R}^n, \mathbb{R}^m)  : \ F\text{ has maximal rank}\},
\end{equation*}
which is an open dense semi-algebraic set in $L(\mathbb{R}^n, \mathbb{R}^m).$  Replace $\mathcal A_2$ by $\mathcal A_2\cap \mathcal A$ if $\mathcal A_2\not\subset \mathcal A.$ Let us prove that $F(X)$ is closed for all $F\in\mathcal A_2$. 
To do this, take arbitrarily $y\in \overline{F(X)},$ we will show that $y\in F(X).$ 
Let $y^k\in F(X)$ be a sequence tending to $y.$ 
By taking a subsequence if necessary, we can assume that either $F^{-1}(y^k)\cap\partial X\ne\emptyset$ for all $k$ or $F^{-1}(y^k)\cap\partial X=\emptyset$ for all $k$.
In the first case, we have $y\in\overline{F(\partial X)}={F(\partial X)}\subset F(X).$
On the other hand, in the second case, we must have $F^{-1}(y^k) \subset X$ for all $k$
(indeed, if the contrary holds, there are $u^k\in F^{-1}(y^k)\cap X$ and $w^k\in F^{-1}(y^k)\setminus X$ for some $k;$
then the segment joining $u^k$ and $w^k$ must contains a point belonging to $F^{-1}(y^k)\cap\partial X,$ which is a contradiction).
Let $x^k$ be the unique point of the intersection $F^{-1}(y^k)\cap (\ker F)^\perp.$
Since $F$ has maximal rank, the restriction $F|_{(\ker F)^\perp}$ is a linear isomorphism from $(\ker F)^\perp$ onto $\mathbb R^m,$ and hence the sequence $x^k$ must converge to a point $x.$ It is clear that $x\in X$ and $F(x)=y.$
Consequently, $F(X)$ is closed.

\subsubsection*{Case $1.2.2$: $d<n$}\

By {\em Case $1.1$}, there is an open dense semi-algebraic set $\mathcal B_1$ in $L(\mathbb R^n,\mathbb R^d)$ such that $F_1(X)$ is closed for all $F_1\in\mathcal B_1$. For $F_1\in\mathcal B_1$, let 
$$\mathcal B_2(F_1):=\{F_2\in L(\mathbb R^d,\mathbb R^m):\ F_2(F_1(X)) \text{ is closed}\}.$$ 
Observe that $\mathcal B_2(F_1)$ contains an open dense semi-algebraic set.
Indeed, this follows from the inductive assumption if $\dim F_1(X)<d$ and from {\em Case $1.2.1$} if $\dim F_1(X)=d$.
Let
$$\mathcal A_3:=\{F\in L(\mathbb R^n,\mathbb R^m):\ \text{ there are } F_1\in \mathcal B_1 \text{ and } F_2\in \mathcal B_2(F_1) \text{ such that }F=F_2\circ F_1\},$$
which is a semi-algebraic set. It is clear that $F(X)$ is closed for all $F\in\mathcal A_3.$
Let us show that $\mathcal A_3$  is dense in $L(\mathbb R^n,\mathbb R^m)$.
For this, pick arbitrarily $F\in L(\mathbb R^n,\mathbb R^m).$
It is not hard to see that there exist $F_1\in L(\mathbb R^n,\mathbb R^d)$ and $F_2\in L(\mathbb R^d,\mathbb R^m)$ such that $F=F_2\circ F_1.$
Since $\mathcal B_1$ is dense in $L(\mathbb R^n,\mathbb R^d)$, there is a sequence $F_1^k\in\mathcal B_1$ such that 
$$\|F_1^k-F_1\|\leqslant \frac{1}{k}.$$
Furthermore, for each $k,$ as $\mathcal B_1(F_1^k)$ is dense in $L(\mathbb R^d,\mathbb R^m)$, there is $F_2^k\in \mathcal B_2(F_1^k)$ such that 
$$\|F_2^k-F_2\|\leqslant \frac{1}{k}.$$
Set $F^k:=F_2^k\circ F_1^k.$ 
We have
$$\begin{array}{lll}\|F^k-F\|&=&\|F_2^k\circ F_1^k-F_2\circ F_1\|\\
&\leqslant& \|F_2^k\circ F_1^k-F_2^k\circ F_1\|+\|F_2^k\circ F_1-F_2\circ F_1\|\\
&\leqslant& \|F_2^k\| \|F_1^k-F_1\|+\|F_2^k-F_2\| \|F_1\|\\
&\leqslant&  \displaystyle \frac{1}{k} \left(\|F_2\| + \frac{1}{k} \right) + \frac{\|F_1\|}{k}.
\end{array}$$
Hence $F^k\to F$ as $k\to\infty.$
Therefore $\mathcal A_3$ is dense in $L(\mathbb R^n,\mathbb R^m).$
By Lemma~\ref{DenseIsOpen}, $\mathcal A_3$ contains an open dense semi-algebraic set in $L(\mathbb R^n,\mathbb R^m).$
This ends the proof for the linear case.

\subsubsection*{Case 2: General case}

Since the set $X$ is closed and the mapping $G$ is continuous, the semi-algebraic set
$$Z := \{(x, y)\in \mathbb{R}^n \times \mathbb{R}^m :\ x \in X,\ y = G(x)\}$$
is closed. Applying {\em Case 1} to $Z,$ we can see that the set
$$\{\mathcal{F} \in L(\mathbb{R}^n \times \mathbb{R}^m, \mathbb{R}^m) :\ \mathcal{F}(Z) \textrm{ is closed} \}$$
contains an open dense semi-algebraic set, say $\mathscr{U} \subset L(\mathbb{R}^n \times \mathbb{R}^m, \mathbb{R}^m).$ 
Consider the mapping
$$\Phi\colon L(\mathbb{R}^n \times \mathbb{R}^m, \mathbb{R}^m)\to L(\mathbb{R}^n, \mathbb{R}^m) \times L(\mathbb{R}^m, \mathbb{R}^m),\ \mathcal F \mapsto (F_1,F_2),$$
where $F_1 \in  L(\mathbb{R}^n, \mathbb{R}^m)$ and $F_2 \in L(\mathbb{R}^m, \mathbb{R}^m)$ are linear mappings defined as follows
$$F_1(x) :=\mathcal F(x,0) \ \textrm{ and } \ F_2(y):=\mathcal F(0,y) \ \textrm{ for all } \ x \in \mathbb{R}^n,\ y \in \mathbb{R}^m.$$
A simple calculation shows that $\Phi$ is a linear isomorphism. Consequently, $\Phi(\mathscr{U})$ is an open dense semi-algebraic set in $L(\mathbb{R}^n, \mathbb{R}^m) \times L(\mathbb{R}^m, \mathbb{R}^m)$. 
In view of Corollary~\ref{Everywhere}, there exists a linear isomorphism $F_2^* \in L(\mathbb{R}^m, \mathbb{R}^m)$ such that  the set
$$\{{F} \in L(\mathbb{R}^n, \mathbb{R}^m) :\ (F, F_2^*) \in \Phi(\mathscr{U}) \}$$
contains an open dense semi-algebraic set in $L(\mathbb{R}^n, \mathbb{R}^m).$ Or, equivalently, the set
$$\{(F_2^*)^{-1} \circ {F} \in L(\mathbb{R}^n, \mathbb{R}^m) :\ (F, F_2^*) \in \Phi(\mathscr{U}) \}$$
contains an open dense semi-algebraic set in $L(\mathbb{R}^n, \mathbb{R}^m).$

Take arbitrarily $(F, F_2^*) \in \Phi(\mathscr{U}).$ Then the image $[\Phi^{-1}(F, F_2^*)](Z)$ is closed, and so is the set $(F_2^*)^{-1} \circ [\Phi^{-1}(F, F_2^*)](Z).$ Observe that
\begin{eqnarray*}
(F_2^*)^{-1} \circ [\Phi^{-1}(F, F_2^*)] (Z)
&=&  \left \{(F_2^*)^{-1} (F(x) + F_2^*(y)) :\ x \in X,\ y = G(x) \right\} \\
&=&  \left \{[(F_2^*)^{-1} \circ F](x) + y :\ x \in X,\ y = G(x) \right\} \\
&=&  \left \{[(F_2^*)^{-1} \circ F](x) + G(x) :\ x \in X \right\}.
\end{eqnarray*}
Hence the last set is also closed. The theorem is proved.
\end{proof}

\subsection*{Acknowledgments}
The first and the third authors were partially supported by the Vietnam Academy of Science and Technology under Grant Number \DD LTE00.01/21-22.
The second author was partially supported by the grant of Narodowe Centrum Nauki number 2019/33/B/ST1/00755. 


\end{document}